\documentclass[12pt]{amsart}
\usepackage{amsmath}
\usepackage{amsfonts}
\usepackage{amssymb}
\usepackage{bbding}
\usepackage{stmaryrd}
\usepackage{txfonts}
\usepackage{graphicx}
\usepackage{epsfig}
\usepackage{xypic}
\usepackage{tikz}
\usepackage{longtable}
\usepackage[all]{xy}
\usepackage{pgflibraryarrows}
\usepackage{pgflibrarysnakes}
\usepackage[shortlabels]{enumitem}
\usepackage{ifpdf}
\ifpdf
  \usepackage[colorlinks,final,backref=page,hyperindex]{hyperref}
\else
  \usepackage[colorlinks,final,backref=page,hyperindex,hypertex]{hyperref}
\fi
\usepackage{xspace}

\newtheorem{theorem}{Theorem}[section]
\newtheorem{lemma}[theorem]{Lemma}

\newtheorem{prop}[theorem]{Proposition}
\theoremstyle{definition}
\newtheorem{defn}[theorem]{Definition}
\newtheorem{remark}[theorem]{Remark}
\newtheorem{exam}[theorem]{Example}

\newcommand{\nc}{\newcommand}
\newcommand{\delete}[1]{}

\nc{\tred}[1]{\textcolor{red}{#1}}
\nc{\tblue}[1]{\textcolor{blue}{#1}} \nc{\tgreen}[1]{\textcolor{green}{#1}} \nc{\tpurple}[1]{\textcolor{purple}{#1}} \nc{\btred}[1]{\textcolor{red}{\bf #1}} \nc{\btblue}[1]{\textcolor{blue}{\bf #1}} \nc{\btgreen}[1]{\textcolor{green}{\bf #1}} \nc{\btpurple}[1]{\textcolor{purple}{\bf #1}}

\newcommand{\efootnote}[1]{}

\nc{\mlabel}[1]{\label{#1}}  
\nc{\mcite}[2][]{\cite[#1]{#2}}  
\nc{\mref}[1]{\ref{#1}}  
\nc{\mbibitem}[1]{\bibitem{#1}} 

\delete{
\nc{\mlabel}[1]{\label{#1}  
{\hfill \hspace{1cm}{\bf{{\ }\hfill(#1)}}}}
\nc{\mcite}[1]{\cite{#1}}  
\nc{\mref}[1]{\ref{#1}{{\bf{{\ }(#1)}}}}  
\nc{\mbibitem}[1]{\bibitem[\bf #1]{#1}} 
}

\renewcommand\geq{\geqslant}
\renewcommand\leq{\leqslant}

\renewcommand\bar[1]{\overline{#1}}

\nc{\nz}{\varepsilon}
\nc{\Id}{\mathrm{Id}}
\nc{\map}[2]{{#2}^{#1}}
\nc{\gp}{B}

\nc{\Irr}{\mathrm{Irr}}
\nc{\vx}{\sigma} \nc{\vy}{\tau} \nc{\dvx}{\sigma^{(1)}} \nc{\dvy}{\tau^{(1)}} \nc{\done}{\vep} \nc{\mcitep}[1]{\mcite{#1}} \nc{\wt}{\mathrm{wt}} \nc{\bre}[1]{|#1|} \nc{\mapmonoid}{\frakM} \nc{\disjoint}{\frakM'}
\nc{\ncpoly}[1]{\langle #1\rangle}  
\nc{\mapm}[1]{\lfloor\!|{#1}|\!\rfloor}
\nc{\diff}[1]{{}^\NC\{ #1 \}} \nc{\disj}[1]{\{{#1}\}'} \nc{\mdisj}[1]{\frakM'(#1)} \nc{\brho}{\bar{\rho}} \nc{\om}{\bar{\frakm}} \nc{\frakn}{\mathfrak n} \nc{\ddeg}[1]{^{(#1)}} \nc{\opset}{X} \nc{\genset}{{Z}} \nc{\NC}{\mathrm{{NC}}} \nc{\leaf}{\mathrm{leaf}} \nc{\twig}{\mathrm{twig}} \nc{\fe}{\mathrm{fl}} \nc{\munderline}[1]{#1} \nc{\bo}{o} \nc{\dep}{\mathrm{depth}} \nc{\ofe}{\mathrm{ofl}} \nc{\dfe}{\mathrm{dfe}} \nc{\fex}{\mathrm{fex}} \nc{\dl}{\mathrm{dlex}} \nc{\db}{\mathrm{db}} \nc{\lex}{\mathrm{lex}} \nc{\clex}{\mathrm{clex}} \nc{\dgp}{\mathrm{dgp}} \nc{\dgx}{\mathrm{dgx}} \nc{\br}{\mathrm{br}} \nc{\obd}{\mathrm{odb}} \nc{\ob}{\mathrm{ob}}
\nc{\pie}{\mathrm{PIE}}
\nc{\rbo}{\mathrm{RBO}}
\nc{\supp}{\mathcal{S}}
\nc{\nul}{\mathcal{Z}}
\nc{\td}{\mathrm{st}}

\nc{\bin}[2]{ (_{\stackrel{\scs{#1}}{\scs{#2}}})}  
\nc{\binc}[2]{ \left (\!\! \begin{array}{c} \scs{#1}\\
    \scs{#2} \end{array}\!\! \right )}  
\nc{\bincc}[2]{  \left ( {\scs{#1} \atop
    \vspace{-1cm}\scs{#2}} \right )}  
\nc{\bs}{\bar{S}} \nc{\cosum}{\sqsubset} \nc{\la}{\longrightarrow} \nc{\rar}{\rightarrow} \nc{\dar}{\downarrow} \nc{\dprod}{**} \nc{\dap}[1]{\downarrow \rlap{$\scriptstyle{#1}$}} \nc{\md}[1]{\bar{#1}} \nc{\uap}[1]{\uparrow \rlap{$\scriptstyle{#1}$}} \nc{\defeq}{\stackrel{\rm def}{=}} \nc{\disp}[1]{\displaystyle{#1}} \nc{\dotcup}{\ \displaystyle{\bigcup^\bullet}\ } \nc{\gzeta}{\bar{\zeta}} \nc{\hcm}{\ \hat{,}\ } \nc{\hts}{\hat{\otimes}} \nc{\barot}{{\otimes}} \nc{\free}[1]{\bar{#1}} \nc{\uni}[1]{\tilde{#1}} \nc{\hcirc}{\hat{\circ}} \nc{\leng}{\ell} \nc{\lleft}{[} \nc{\lright}{]} \nc{\lc}{\lfloor} \nc{\rc}{\rfloor}
\nc{\lb}{[} 
\nc{\rb}{]} 
\nc{\curlyl}{\left \{ \begin{array}{c} {} \\ {} \end{array}
    \right.  \!\!\!\!\!\!\!}
\nc{\curlyr}{ \!\!\!\!\!\!\!
    \left. \begin{array}{c} {} \\ {} \end{array}
    \right \} }
\nc{\longmid}{\left | \begin{array}{c} {} \\ {} \end{array}
    \right. \!\!\!\!\!\!\!}
\nc{\onetree}{\bullet} \nc{\ora}[1]{\stackrel{#1}{\rar}}
\nc{\ola}[1]{\stackrel{#1}{\la}}
\nc{\ot}{\otimes} \nc{\mot}{{{\boxtimes\,}}} \nc{\otm}{\overline{\boxtimes}} \nc{\sprod}{\bullet} \nc{\scs}[1]{\scriptstyle{#1}} \nc{\mrm}[1]{{\rm #1}} \nc{\msum}{\sum\limits}
\nc{\margin}[1]{\marginpar{\rm #1}}   
\nc{\dirlim}{\displaystyle{\lim_{\longrightarrow}}\,} \nc{\invlim}{\displaystyle{\lim_{\longleftarrow}}\,} \nc{\mvp}{\vspace{0.3cm}} \nc{\tk}{^{(k)}} \nc{\tp}{^\prime} \nc{\ttp}{^{\prime\prime}} \nc{\svp}{\vspace{2cm}} \nc{\vp}{\vspace{8cm}} \nc{\proofbegin}{\noindent{\bf Proof: }}
\nc{\proofend}{$\blacksquare$ \vspace{0.3cm}}
\nc{\modg}[1]{\!<\!\!{#1}\!\!>}
\nc{\intg}[1]{F_C(#1)} \nc{\lmodg}{\!<\!\!} \nc{\rmodg}{\!\!>\!} \nc{\cpi}{\widehat{\Pi}}
\nc{\sha}{{\mbox{\cyr X}}}  
\nc{\shap}{{\mbox{\cyrs X}}} 
\nc{\shpr}{\diamond}    
\nc{\shp}{\ast} \nc{\shplus}{\shpr^+}
\nc{\shprc}{\shpr_c}    
\nc{\msh}{\ast} \nc{\zprod}{m_0} \nc{\oprod}{m_1} \nc{\vep}{\varepsilon} \nc{\labs}{\mid\!} \nc{\rabs}{\!\mid}
\nc{\astarrow}{\overset{\raisebox{-3pt}{$\ast$}}{\rightarrow}}

\nc{\sym}{\mrm{Sym}}
\nc{\qsym}{\mrm{QSym}}
\nc{\syms}{symmetric functions\xspace}
\nc{\eqsym}{extended quasisymmetric function\xspace}
\nc{\eqsyms}{extended quasisymmetric functions\xspace}
\nc{\Eqsyms}{Extended Quasisymmetric functions\xspace}
\nc{\Esyms}{Extended symmetric functions\xspace}
\nc{\EQSYM}{\mrm{EQSym}}
\nc{\sgqsym}{quasisymmetric function with semigroup exponents\xspace}
\nc{\sgqsyms}{quasisymmetric functions with semigroup exponents\xspace}
\nc{\Sgqsyms}{Quasisymmetric functions with semigroup exponents\xspace}
\nc{\SGQSYM}{\mrm{SGQSYM}}
\nc{\emzv}{extended multiple zeta value}
\nc{\emzvs}{extended multiple zeta values}
\nc{\sgfps}{formal power series with semigroup exponent\xspace}
\nc{\nsymg}{\mathrm{NSym}_\gp}

\nc{\parr}{\rm Par}
\nc{\wpar}{\rm WPar}

\nc{\dth}{d} \nc{\mmbox}[1]{\mbox{\ #1\ }} \nc{\fp}{\mrm{FP}} \nc{\rchar}{\mrm{char}} \nc{\Fil}{\mrm{Fil}} \nc{\Mor}{Mor\xspace} \nc{\gmzvs}{gMZV\xspace} \nc{\gmzv}{gMZV\xspace} \nc{\mzv}{MZV\xspace} \nc{\mzvs}{MZVs\xspace} \nc{\Hom}{\mrm{Hom}} \nc{\id}{\mrm{id}} \nc{\im}{\mrm{im}} \nc{\incl}{\mrm{incl}}  \nc{\mchar}{\rm char}

\nc{\Alg}{\mathbf{Alg}} \nc{\Bax}{\mathbf{Bax}} \nc{\bff}{\mathbf f} \nc{\bfk}{{\bf k}} \nc{\bfone}{{\bf 1}} \nc{\bfx}{\mathbf x} \nc{\bfy}{\mathbf y}
\nc{\base}[1]{\bfone^{\otimes ({#1}+1)}} 
\nc{\Cat}{\mathbf{Cat}} \delete{}
\nc{\detail}{\marginpar{\bf More detail}
    \noindent{\bf Need more detail!}
    \svp}
\nc{\Int}{\mathbf{Int}} \nc{\Mon}{\mathbf{Mon}}
\nc{\rbtm}{{shuffle }} \nc{\rbto}{{Rota-Baxter }} \nc{\remarks}{\noindent{\bf Remarks: }} \nc{\Rings}{\mathbf{Rings}} \nc{\Sets}{\mathbf{Sets}}
\nc{\balpha}{\mathbf{\alpha}}

\nc{\BA}{{\mathbb A}} \nc{\CC}{{\mathbb C}} \nc{\DD}{{\mathbb D}} \nc{\EE}{{\mathbb E}} \nc{\FF}{{\mathbb F}} \nc{\GG}{{\mathbb G}} \nc{\HH}{{\mathbb H}} \nc{\LL}{{\mathbb L}} \nc{\NN}{{\mathbb N}} \nc{\KK}{{\mathbb K}} \nc{\PP}{{\mathbb P}} \nc{\QQ}{{\mathbb Q}} \nc{\RR}{{\mathbb R}} \nc{\TT}{{\mathbb T}} \nc{\VV}{{\mathbb V}} \nc{\ZZ}{{\mathbb Z}}

\nc{\cala}{{\mathcal A}} \nc{\calc}{{\mathcal C}} \nc{\cald}{{\mathcal D}} \nc{\cale}{{\mathcal E}} \nc{\calf}{{\mathcal F}} \nc{\calg}{{\mathcal G}} \nc{\calh}{{\mathcal H}} \nc{\cali}{{\mathcal I}} \nc{\call}{{\mathcal L}} \nc{\calm}{{\mathcal M}} \nc{\caln}{{\mathcal N}} \nc{\calo}{{\mathcal O}} \nc{\calp}{{\mathcal P}} \nc{\calr}{{\mathcal R}} \nc{\cals}{{\mathcal S}} \nc{\calt}{{\mathcal T}} \nc{\calw}{{\mathcal W}} \nc{\calk}{{\mathcal K}} \nc{\calx}{{\mathcal X}}
\nc{\calz}{{\mathcal Z}}
 \nc{\CA}{\mathcal{A}}

\nc{\fraka}{{\mathfrak a}} \nc{\frakA}{{\mathfrak A}} \nc{\frakb}{{\mathfrak b}} \nc{\frakB}{{\mathfrak B}}
\nc{\frakc}{{\mathfrak c}}  \nc{\frakD}{{\mathfrak D}}
\nc{\frakH}{{\mathfrak H}}
\nc{\frakh}{{\mathfrak h}} \nc{\frakM}{{\mathfrak M}}
\nc{\frakO}{{\mathfrak O}}
\nc{\frakE}{{\mathfrak E}}
\nc{\bfrakM}{\overline{\frakM}} \nc{\frakm}{{\mathfrak m}} \nc{\frakP}{{\mathfrak P}} \nc{\frakN}{{\mathfrak N}} \nc{\frakp}{{\mathfrak p}} \nc{\frakS}{{\mathfrak S}}
\nc{\frakk}{{\mathfrak k}}
\nc{\frakx}{{\mathfrak x}}
\nc{\frakl}{{\mathfrak l}} \nc{\ox}{\bar{\frakx}} \nc{\frakX}{{\mathfrak X}} \nc{\fraky}{{\mathfrak y}} \nc\dop{\delta}
\nc{\Reduce}{{\rm Red}}

\font\cyr=wncyr10 \font\cyrs=wncyr7
\nc{\redt}[1]{\textcolor{red}{#1}}
\nc{\li}[1]{\textcolor{red}{\tt Li:#1}}
\nc{\yu}[1]{\textcolor{blue}{\tt Yu:#1}}

\begin{document}
\title{Characteristics of Rota-Baxter Algebras}

\author{Li Guo}
\address{
Department of Mathematics and Computer Science, Rutgers University, Newark, NJ 07102, USA}
\email{liguo@rutgers.edu}

\author{Houyi Yu}
\address{School  of Mathematics and Statistics, Southwest University, Chongqing, China}
\email{yuhouyi@swu.edu.cn}

\hyphenpenalty=8000
\date{\today}

\begin{abstract}
The characteristic is a simple yet important invariant of an algebra. In this paper, we study the characteristic of a Rota-Baxter algebra, called the Rota-Baxter characteristic. We introduce an invariant, called the ascent set, of a Rota-Baxter characteristic. By studying its properties, we classify Rota-Baxter characteristics in the homogenous case and relate Rota-Baxter characteristics in general to the homogeneous case through initial ideals. We show that
the Rota-Baxter quotients of Rota-Baxter characteristics have the same underlying sets as those in the homogeneous case.
We also give a more detailed study of Rota-Baxter characteristics with special base rings. In particular, we determine the prime characteristics of Rota-Baxter rings.
\end{abstract}

\subjclass[2010]{13C05, 13E05, 16W99}

\keywords{Rota-Baxter algebra, Rota-Baxter characteristic, initial object, initial ideal, prime ideal}

\maketitle

\tableofcontents

\hyphenpenalty=8000 \setcounter{section}{0}


\allowdisplaybreaks

\section{Introduction}\label{Introduction}

A Rota-Baxter algebra is an associative algebra together
with a linear endomorphism that is an algebraic analogue of the integral operator.
This concept has it origin from a work of G. Baxter~\cite{Baxter1960} in probability theory.
In the late 1960s, Rota \cite{Ro1} studied the subject from
an algebraic and combinatorial perspective and suggested that they are closely related to
hypergeometric functions, incidence algebras and symmetric functions \cite{Ro2,Ro1998}.
Since then, these algebras have been investigated by mathematicians and mathematical physicists with various motivations. For example, a Rota-Baxter operator on a Lie algebra is closely related to the operator form of the classical Yang-Baxter equation. Here the Baxter is the physicist R. Baxter~\cite{Bai,STS}. In recent years Rota-Baxter operators have found applications to many areas, such as number theory \cite{guozhang2008}, combinatorics \cite{GG,Ro1,Ro2}, operads \cite{AL,Baibguoni2013} and
quantum field theory \cite{CK1998}. See \cite{guo2012,guo2009} and
the references therein for further details.

For the theoretic study of this important algebraic structure, it is useful to generalize the study of characteristics of algebras to Rota-Baxter algebras. The characteristic of a (unitary) ring $R$ is (the nonnegative generator of) the kernel of the structure map $\ZZ\to R$ sending $1$ to the identity element of $R$. More generally, the characteristic of an algebra $R$ over a commutative ring $\bfk$ is the kernel of the structure map $\bfk\to R$. To generalize this concept to the context of Rota-Baxter algebras, we note that the structure map $\bfk\to R$ comes from the fact that $\bfk$ is the initial object in the category of $\bfk$-algebras, as the free $\bfk$-algebra on the empty set. Then to study the characteristics of Rota-Baxter $\bfk$-algebras, we consider the initial object in the category of Rota-Baxter $\bfk$-algebras, as the free Rota-Baxter algebra on the empty set. Thus every Rota-Baxter $\bfk$-algebra comes with a (unique) structure map from the initial object to this Rota-Baxter algebra. Then the kernel of the structure map should be the characteristic of the Rota-Baxter algebra.

In this paper, we study the characteristics of Rota-Baxter algebras. More precisely, we study the Rota-Baxter ideals, and their corresponding quotients, of the initial Rota-Baxter $\bfk$-algebra. In particular, we study the Rota-Baxter ideals of the initial Rota-Baxter ring. The initial Rota-Baxter algebra is a generalization of the divided power algebra, and the polynomial algebra $\bfk[x]$ when $\bfk$ is taken to contain $\QQ$. So some our results are naturally related to results of these algebras.

The construction of the initial object will be reviewed at the beginning of Section~\ref{Rbisection} and will be applied to give the concept of Rota-Baxter characteristics. In Section \ref{Rbi2section3}, we first classify homogeneous Rota-Baxter characteristics and their quotients (Theorem~\ref{corhomoideal}). We then show that the Rota-Baxter quotients of general Rota-Baxter characteristics have the same underlying sets as those in the homogeneous case~(Theorem \ref{mainthmdcomq}).
Finally specializing to the case of $\bfk=\ZZ$ in Section~\ref{k=z3},
we give the structures of Rota-Baxter characteristics of Rota-Baxter rings, and determine the prime Rota-Baxter characteristics (Theorem \ref{primesirb}).

\section{Characteristics of Rota-Baxter algebras}\label{Rbisection}

\noindent
{\bf Notations.} Unless otherwise specified, an algebra in this paper is assumed to be unitary associative defined over a unitary commutative ring $\bfk$.
Let $\NN$ and $\PP$ denote the set of nonnegative and positive integers respectively. For $n\in \NN$, denote $[n]:=\{1,\cdots,n\}$. For notational convenience, we also denote $[\infty]=\PP$. Let $\ZZ_n$ denote the set of integers modulo $n$.

\smallskip

Before giving the definition of the characteristic of a Rota-Baxter algebra, we first provide some background and preliminary results on Rota-Baxter algebras. See~\mcite{EG,guo2012,guokeigher20001} for details.

Let $\lambda\in \bfk$ be given. A {\bf Rota-Baxter $\bfk$-algebra} of weight $\lambda$ is a
$\bfk$-algebra $R$ paired with a linear operator $P$ on $R$ that satisfies the identity
\begin{equation}\label{rtequ}
P(x)P(y)=P(xP(y))+P(P(x)y)+\lambda P(xy)
\end{equation}
for all $x,y\in R$. When $\bfk$ is $\ZZ$, then the pair $(R,P)$ is called a {\bf Rota-Baxter ring}.

A {\bf Rota-Baxter ideal} of a Rota-Baxter algebra $(R,P)$ is an ideal $I$ of $R$ such that $P(I)\subseteq I$. Then we denote $I\leq R$.
The concepts of Rota-Baxter subalgebras, quotient Rota-Baxter algebras and homomorphisms
of Rota-Baxter algebras can be similarly defined.

We recall that the {\bf free Rota-Baxter algebra on a set $X$} is a Rota-Baxter algebra $(F_{RB}(X),P_X)$ together with a set map $i_X:X\to F_{RB}(X)$ characterized by the universal property that, for any Rota-Baxter algebra $(R,P)$ and set map $f:X\to R$, there is a unique Rota-Baxter algebra homomorphism $\tilde{f}:F_{RB}(X)\to R$ such that $i_X\circ \tilde{f}=f$.
The free Rota-Baxter algebra on the empty set is also the free Rota-Baxter $\mathbf{k}$-algebra $F_{RB}(\bfk)$ on the $\bfk$-algebra $\mathbf{k}$~\mcite{guokeigher20001}, characterized by the universal property that, for any Rota-Baxter algebra $(R,P)$ and $\bfk$-algebra homomorphism $f:\bfk\to R$, there is a unique Rota-Baxter algebra homomorphism $\tilde{f}:F_{RB}(\bfk)\to R$ such that $i_\bfk\circ \tilde{f}=f$, where $i_\bfk$ is the structure map $i_\bfk:\bfk\to F_{RB}(\bfk)$. Since there is only one $\bfk$-algebra homomorphism $\bfk\to R$, this universal property shows that $F_{RB}(\bfk)$ is the initial object in the category of Rota-Baxter $\bfk$-algebras.

We refer the reader to~\cite{guokeigher20001}, as well as~\cite{C1972,Ro1}, for the general constructions of free Rota-Baxter algebras, but will focus on a simple construction of $F_{RB}(\bfk)$ following~\cite{AGKO}, where it is denoted by $\sha_\bfk (\bfk)$.
This free  Rota-Baxter $\mathbf{k}$-algebra not only provides the simplest example of free Rota-Baxter  algebras but also establishes the connection between Rota-Baxter algebra and some well-known concepts such as divided powers and
Stirling numbers~\mcite{AGKO,Gust}.
As a $\bfk$-module, $\sha_{\mathbf{k} }(\mathbf{k})$ is given by the free $\bfk$-module
\begin{align*}
\sha_{\mathbf{k} }(\mathbf{k})=\bigoplus_{m=0}^{\infty}\mathbf{k} \fraka_{m}
\end{align*}
on the basis $\left\{\fraka_m\,|\, m\geq 0\right\}$.
For a given $\lambda\in\bfk$, the product $\diamond=\diamond_\lambda$ on $\sha_{\mathbf{k}}(\mathbf{k})$ is defined by
\begin{align}\label{productformula1}
\fraka_{m}\diamond_{\lambda}\fraka_{n}
=\sum_{i=0}^{{\rm min}(m,n)}\binom{m+n-i}{m}\binom{m}{i}\lambda^i\fraka_{m+n-i}, \qquad m,n\in \mathbb{N}.
\end{align}
Thus when $\lambda=0$, $\sha_\bfk(\bfk)$ is the divided power algebra.
Note that $\diamond$ is an extension of the product on $\bfk$ viewed as $\bfk \fraka_0$. Thus there should no confusion if the notation $\diamond$ is suppressed, as we often do.
Define the $\bfk$-linear operator $P=P_{\mathbf{k}}$ on $\sha_{\mathbf{k}}(\mathbf{k})$  by assigning
$P_{\mathbf{k}}(\fraka_m)=\fraka_{m+1}$, $m\geq0$
and extending by additivity.

By~\mcite{guokeigher20001}, when $\bfk$ contains $\QQ$ and $\lambda=0$, we have $\sha_\bfk(\bfk)\cong\bfk[x]$ as a $\bfk$-algebra.

\begin{theorem}\label{freerbak}
The pair $(\sha_{\bfk}(\bfk),P_\bfk)$ is the initial object in the category of Rota-Baxter $\bfk$-algebras.
More precisely, for any Rota-Baxter $\bfk$-algebra $(R, P)$, there is a unique
Rota-Baxter $\bfk$-algebra homomorphism $\varphi: (\sha_{\bfk}(\bfk), P_{\bfk})\rightarrow(R, P)$.
\end{theorem}

\begin{proof}
Suppose a $\bfk$-algebra $A$ has a linear basis $X$. By~\cite{EG,guo2012}, the free Rota-Baxter algebra $F_{RB}(A)$ on $A$ (denoted by $\sha^{\mathrm{NC}}_\bfk(A)$) has a linear basis consisting of Rota-Baxter words in the alphabet set $X$. By definition, a Rota-Baxter word in $X$ is a bracketed word in the alphabet set $X$ in which there are no adjacent pairs of brackets. Taking $A=\bfk$, then $X=\{1\}$. Thus a Rota-Baxter word in $X$ can only be of the form
$$u_r:=\underbrace{\lc\cdots\lc}_{r\ \text{iterations}} 1\underbrace{\rc\cdots \rc}_{r\ \text{iterations}}$$
(namely applying the bracket operator $\lc\ \rc$ to $1$ $r$ times), for $r\geq 1$, together with $u_0:=1$. Thus
$$F_{RB}(\bfk)=\sum_{r\geq 0} \bfk u_r,$$
on which the Rota-Baxter operator $Q_\bfk$ is given by $Q_\bfk(u_r)=\lc u_r\rc=u_{r+1}$. By the universal property of $F_{RB}(\bfk)$, the natural inclusion map $f:\bfk \to \sha_\bfk(\bfk)$ sending $1\to \fraka_0$ (that is, the structure map) extends to a Rota-Baxter algebra homomorphism
$$\tilde{f}: F_{RB}(\bfk) \to \sha_\bfk(\bfk).$$
As such, we obtain $\tilde{f}(u_0)=\fraka_0$ and recursively,
$$ \tilde{f}(u_{r+1})=\tilde{f}(Q_\bfk(u_r))=P_\bfk(\tilde{f}(u_r)) =P_\bfk(\fraka_r)=\fraka_{r+1}, \quad r\geq 0.$$
Therefore $\tilde{f}$ is a linear isomorphism and thus a Rota-Baxter algebra isomorphism, showing that $\sha_\bfk(\bfk)$ is the free Rota-Baxter algebra on $\bfk$ and hence the initial object in the category of Rota-Baxter $\bfk$-algebras.
\end{proof}

Thus $\sha_{\mathbf{k}}(\bfk)$ plays the same role in the category of Rota-Baxter $\mathbf{k}$-algebras as the role played by $\bfk$ in the category of $\bfk$-algebras.
\begin{defn}
Let $(R,P)$ be a Rota-Baxter $\bfk$-algebra and let $\varphi=\varphi_{(R,P)}:(\sha_\bfk(\bfk),P_\bfk)\to (R,P)$ be the unique Rota-Baxter algebra homomorphism from the initial object $(\sha_\bfk(\bfk),P_\bfk)$ in the category of Rota-Baxter $\bfk$-algebras to $(R,P)$.
The kernel of the structure map $\varphi$ is called the {\bf Rota-Baxter characteristic} of $(R,P)$.
\end{defn}

In view of Theorem \ref{freerbak},  the characteristic of a Rota-Baxter $\bfk$-algebra $R$ must be a Rota-Baxter ideal of $\sha_{\bfk}(\bfk)$. Conversely, any Rota-Baxter ideal $I$ of $\sha_\bfk(\bfk)$ is the characteristic of some Rota-Baxter algebra, for example of $\sha_\bfk(\bfk)/I$.

\begin{remark}
Based on the above observation, we will use the terminology Rota-Baxter characteristics in exchange with Rota-Baxter ideals of $\sha_\bfk(\bfk)$ in the rest of the paper.
\end{remark}

\section{Classification of the characteristics of Rota-Baxter algebras}\label{Rbi2section3}

In this section, we study Rota-Baxter characteristics and their quotients. First in Section~\mref{ss:ascend}, we introduce an invariant, called the ascent set, of a Rota-Baxter characteristic. We then apply the invariants to classify all homogeneous Rota-Baxter characteristics in Section~\mref{ss:homog}. Finally in Section~\mref{ss:general} we relate a Rota-Baxter characteristic to a homogeneous Rota-Baxter characteristic by taking the initial terms and show that the quotients of the two Rota-Baxter characteristics share the same underlying sets, but not the same $\bfk$-modules. This approach is motivated from the study in the polynomial algebra $\bfk[x_1,\cdots,x_n]$ when $\bfk$ is a field. There the quotient modulo an ideal and the quotient modulo the corresponding initial ideal are known to share the same basis~\cite[\S~5.3, Proposition~4]{CLO98}.

\subsection{The ascent set of a Rota-Baxter characteristic}
\mlabel{ss:ascend}

As recalled in the last section, the Rota-Baxter algebra $\sha_\bfk(\bfk)$, as the initial object in the category of Rota-Baxter algebras, is the direct sum
$$ \sha_\bfk(\bfk)=\bigoplus_{m\geq 0} \bfk \fraka_m$$
on the basis $\{\fraka_m\,|\,m\geq 0\}$ and hence is an $\NN$-graded $\bfk$--module. Any nonzero element $f$ of $\sha_\bfk(\bfk)$ can be uniquely written as $f=\sum\limits_{i=0}^n c_i\fraka_{i}$ with $c_n\neq 0$. Then $n$ is called the {\bf degree} of $f$ and $c_n\fraka_n$ is called the {\bf initial term} of $f$, denoted by $\deg f$ and ${\rm in}(f)$ respectively. In addition, we define $\deg 0=-\infty$. We call the set ${\rm supp}(f):=\{c_i\fraka_{i}|c_i\neq 0\}$ the {\bf support} of $f$.

Let $I$ be a Rota-Baxter characteristic, namely an ideal of $\sha_\bfk(\bfk)$. We introduce an invariant of $I$.
For each $j\in \mathbb{N}$, we denote
$$
\Omega_{j}:=\Omega_{j}(I):=\left\{b_j\in\mathbf{k}\,\left|\,(\exists f\in I)\ f=\sum\limits_{i=0}^jb_i\fraka_{i}\right.\right\}.
$$
Equivalently,
$$
\Omega_{j}=\{0\}\cup\left\{b_j\in\mathbf{k}\,\left|\,(\exists f\in I)\ {\rm{in}}(f)=b_j\fraka_{j}\right.\right\}.
$$
The smallest index $j$ such that $\Omega_j\neq 0$ is called the {\bf starting point} of $I$, denoted by $\td(I)$, that is,
\begin{align*}
\td(I)=\min\left\{j\in\NN\,\left|\,\Omega_j\neq0\right.\right\}.
\end{align*}

\begin{lemma}\label{idealisxi}
Let $I$ be a Rota-Baxter ideal of $\sha_{\bfk}(\bfk)$. Then for each $j\in \mathbb{N}$, $\Omega_j$ is an ideal of $\bfk$ and $\Omega_j\subseteq\Omega_{j+1}$.
\end{lemma}

\begin{proof}
For any $b_j,c_j\in\Omega_j$, there exist $f$ and $g\in I$ such that $f=\sum\limits_{i=0}^jb_i\fraka_{i}$
and
$g=\sum\limits_{i=0}^jc_i\fraka_{i}$. Hence $f+g=\sum\limits_{i=0}^j(b_i+c_i)\fraka_{i}$ is in $I$.
So $b_j+c_j\in \Omega_j$. On the other hand, for any $c\in \mathbf{k}$,
$cf=\sum\limits_{i=0}^jcb_i\fraka_{i}\in I$. This yields $cb_j\in \Omega_j$. Thus $\Omega_j$ is an ideal of $\mathbf{k}$.

Since $I$ is a Rota-Baxter ideal of $\sha_{\bfk}(\bfk)$,  we have $P(f)=\sum\limits_{i=0}^jb_i\fraka_{i+1}\in I$. So $b_j\in\Omega_{j+1}$. Hence $\Omega_j\subseteq\Omega_{j+1}$, as required.
\end{proof}

For a given Rota-Baxter ideal $I$ of $\sha_{\bfk}(\bfk)$, Lemma~\ref{idealisxi} shows that the ideals $\Omega_j\subseteq \bfk, j\in \NN$, form a non-decreasing sequence of ideals of $\bfk$. Thus the sequence is controlled by the locations and extents where the increases occur. This motivates us to introduce the following notions.

Let
$s_1< s_2<\cdots$ be the integers $t\in \NN$ such that $\Omega_{t-1}\subsetneq \Omega_{t}$ with the convention that $\Omega_{-1}=\left\{0\right\}$.
We will use the notation $s_i, i\in [N_I]$ where $N=N_I$ is either in $\NN$ or $\infty$ with the convention adopted at the beginning of the last section.
The integers $s_i, i\in [N_I],$ are called the {\bf ascending points} of $I$ and the ideals $\Omega_{s_i}, i\in [N_I]$ are called the {\bf ascending levels} of $I$. In view of Lemma~\ref{idealisxi}, we have $s_1=\td(I)$.

Thus for a given Rota-Baxter characteristic $I$, the set of pairs
\begin{equation}
A(I):=\left\{\left . (s_j,\Omega_{s_j})\,\right|\, j\in [N_I]\right\},
\label{eq:elev}
\end{equation}
called the {\bf ascent set}, is uniquely determined by $I$.
By the definition of $\Omega_{s_j}$, we have
${s_j}<{s_{j+1}}$ and
\begin{align}\label{idealasscend}
\Omega_{{s_j}}=\Omega_{{s_j}+1}=\cdots=\Omega_{s_{j+1}-1}\subsetneq \Omega_{{s_{j+1}}}
\end{align}
for all $j\in [N_I]$, as illustrated by Figure $1$.

\begin{center}
\setlength{\unitlength}{1mm}
\begin{picture}(120,67)(-5,-5)
\put(2,5){\vector(1,0){110}}
\put(2,5){\vector(0,1){45}}
\put(118,4){$\NN$}
\put(0.5,53){$\Omega$}
\put(2,0){$0$}
\linethickness{1.5pt}
\put(2,5){\line(1,0){15.2}}
\put(17,5){\circle*{1}}
\put(15,0){$s_1$}
\put(2,15){\circle*{1}}
\put(-5,14){$\Omega_{s_1}$}
\thicklines
\multiput(17,5)(0,2){5}{\line(0,1){1}}
\linethickness{1.5pt}
\put(17,15){\line(1,0){20.2}}
\put(37,5){\circle*{1}}
\put(35,0){$s_2$}
\put(2,22){\circle*{1}}
\put(-5,21){$\Omega_{s_2}$}
\thicklines
\multiput(37,15)(0,2){4}{\line(0,1){1}}
\linethickness{1.5pt}
\put(37,22){\line(1,0){12.2}}
\put(49,5){\circle*{1}}
\put(47,0){$s_3$}
\put(2,38){\circle*{1}}
\put(-5,37){$\Omega_{s_3}$}
\thicklines
\multiput(49,22)(0,2){8}{\line(0,1){1}}
\linethickness{1.5pt}
\put(49,38){\line(1,0){30}}
\put(47,-8){Figure 1.}
\end{picture}
\end{center}

\vspace{3mm}

We next study how the information from $A(I)$ can be used to recover $I$.
Denote
\begin{equation}
\cala:=\left\{ \left.\{(s_j,\Omega_{s_j})\}_{j\in [N]} \,\right|\, s_j\in \NN, \Omega_{s_j}\leq \bfk, s_j<s_{j+1}, \Omega_{s_j}\subsetneq \Omega_{s_{j+1}}, j\in [N], 1\leq N\leq \infty\right\},
\label{eq:ascpair}
\end{equation}
called the set of {\bf ascending pairs}. So $\cala$ consists of pairs of sequences of the same lengths with one sequence of strictly increasing nonnegative integers and a second sequence of strictly increasing ideals of $\bfk$. Let $\cali=\cali(\bfk)$ denote the set of Rota-Baxter characteristics, namely the set of Rota-Baxter ideals of $\sha_\bfk(\bfk)$. Then taking the ascent set of a Rota-Baxter characteristic defines a map
\begin{equation}
\Phi: \cali \to \cala, \quad I\mapsto A(I), I\leq \sha_\bfk(\bfk).
\mlabel{eq:phi}
\end{equation}

In the rest of the paper, we study the property of this map, including its surjectivity and fibers, that is, inverse images. In Theorem~\ref{corhomoideal}, we show that the restriction of $\Phi$ to the subset of homogeneous Rota-Baxter characteristics gives a bijection to $\cala$, proving the surjectivity of $\Phi$. In Proposition~\ref{inid}, we show that two Rota-Baxter characteristics are in the same fiber if and only if they have the same initial ideal.

\begin{lemma}\label{generatingset2.3}
Let $I$ be a Rota-Baxter ideal of $\sha_\bfk(\bfk)$ with $A(I)=\left\{ (s_j,\Omega_{s_j})\,|\, j\in [N_I]\right\}$.
For each $j\in[N_I]$, we let $\Theta_j:=\left\{\omega_{s_j,\ell}\,\left|\,\ell\in [N_j]\right.\right\}$ be a set of generators of the ideal
$\Omega_{s_j}$. Here $N_j$ is either a positive integer or $\infty$. For each $\omega_{s_j,\ell}\in \Theta_j$, let $f_{s_j,\ell}$ be an element of $I$ whose initial term is $\omega_{s_j,\ell}\fraka_{s_j}$.
Then the set
$$\bigcup_{j\in [N_I]}\left\{\left. f_{s_j,\ell}\,\right|\,  \ell\in [N_j]\right\}$$
is a generating set of the Rota-Baxter ideal $I$.
\label{lem:gengen}
\end{lemma}
We call the  above-mentioned set $\bigcup\limits_{j\in [N_I]}\left\{\left. f_{s_j,\ell}\,\right|\,  \ell\in [N_j]\right\}$ an {\bf ascent generating set}
of the Rota-Baxter ideal $I$.

\begin{proof}
Let $I'$ be the Rota-Baxter ideal generated by the set $\bigcup\limits_{j\in [N_I]}\left\{\left. f_{s_j,\ell}\,\right|\,  \ell\in [N_j]\right\}$.
Since $I'\subseteq I$ is trivial, it suffices to show that each element $f$ of $I$ belongs to $I'$.

So take an $f\in I$. Then $f$ is in $I'$ if $f=0$.
We next assume that $f\neq 0$. Then we have ${\rm in}(f)=b\fraka_{\deg f}$ for some $0\neq b\in \bfk$.
We proceed by induction on $\deg f$. Clearly, $\deg f\geq \td(I)=s_1$.

If $\deg f=s_1$, then $b\in \Omega_{s_1}$. So there exist $c_{s_1,\ell}\in \bfk, \ell\in [N_1],$ all but finitely many of which being zero, such that $b=\sum\limits_{\ell\in [N_1]}c_{s_1,\ell}\omega_{s_1,\ell}$.
Thus the element
$$g=f-\sum\limits_{\ell\in [N_1]}c_{s_1,\ell}f_{s_1,\ell}$$
is in $I$. But now the degree of $g$ is less than $\deg f$. It follows from $\deg f=\td(I)$ that $g$ must be $0$,
which means that $f=\sum\limits_{\ell\in[N_1]}c_{s_1,\ell}f_{s_1,\ell}$ is in $I'$ and we are done. For a given $n\geq \td(I)$,
assume that all $f\in I$ with $\deg f\leq n$ are in $I'$ and take $f\in I$ with $\deg f=n+1$.
Then there exists $r\in [N_I]$ such that $s_{r}\leq n+1<s_{r+1}$ with the convention that $s_{N_I+1}=\infty$ if $N_I$ is finite.
By Eq. \eqref{idealasscend}, we have $\Omega_{n+1}=\Omega_{s_r}$. So $b=\sum\limits_{\ell\in [N_r]}c_{s_r,\ell}\omega_{s_r,\ell}$ where $c_{s_r,\ell}\in \bfk, \ell\in [N_r]$,
with all but a finite number of $c_{s_r,\ell}$ being zero.
Put
$$
h=f-\sum\limits_{\ell\in[N_r]} c_{s_r,\ell}P^{n+1-s_r}\left(f_{s_r,\ell}\right).
$$
Then $h\in I$ with $\deg h<\deg f$. So we can apply the induction hypothesis to obtain that $h$ is in $I'$ and hence $f$ is in $I'$, as required.
\end{proof}

\subsection{Classification of homogeneous Rota-Baxter characteristics}
\mlabel{ss:homog}
We now apply ascent sets to classify homogeneous Rota-Baxter characteristics.

\begin{defn}
A Rota-Baxter ideal $I$ of $\sha_{\bfk}(\bfk)$ is called a {\bf homogeneous Rota-Baxter ideal} if $I$ is a Rota-Baxter ideal generated by a set of homogeneous elements. Then the Rota-Baxter characteristic $I$ is called {\bf homogeneous}.
\end{defn}

We next show that, for a Rota-Baxter ideal of $\sha_\bfk(\bfk)$, its homogeneity as a Rota-Baxter ideal is equivalent to its homogeneity as an ideal. For this purpose, we first give a general relation between generating a Rota-Baxter ideal and generating an ideal in a Rota-Baxter algebra.

\begin{lemma}\label{rbigensetlem}
Let $(R,P)$ be a Rota-Baxter  $\mathbf{k}$-algebra of weight $\lambda$, and $S$ a subset of $R$.
Then the Rota-Baxter ideal  $(S)_{RB}$ generated by $S$ is the ideal of $R$ generated by the set
\begin{align}\label{rtgeseteq}
S_{RB}:=\bigcup_{m\in\mathbb{N}}\left\{(\circ_{i=1}^mP_{x_{i}})(a)|a\in S, x_{i}\in R, 1\leq i\leq m\right\},
\end{align}
where $(\circ_{i=1}^mP_{x_{i}})(a):=P(x_{m}P(x_{m-1}P(\cdots P(x_{1}a))))$ with the convention that $(\circ_{i=1}^0P_{x_{i}})(a):=a$.

If $R=\sha_{\bfk}(\bfk)$, then it suffices to take $x_{i}$ to be the homogeneous elements in Eq.~\eqref{rtgeseteq}.
\end{lemma}

\begin{proof}
Let $I$ be the ideal of $R$ generated by $S_{RB}$. Then an element of $I$ is a sum of elements of the form $r(\circ_{i=1}^mP_{x_{i}})(a)$ with $r\in R$. Then we have
\begin{align*}
P(r(\circ_{i=1}^mP_{x_{i}})(a))=(\circ_{i=1}^{m+1}P_{x_{i}})(a),
\end{align*}
where $x_{m+1}=r$. So, by the additivity of $P$, we obtain $P(I)\subseteq I$, and hence $I$ is the Rota-Baxter ideal of $R$  generated by $S_{RB}$,
that is, $I=( S_{RB})_{RB}$.
Notice that $S\subseteq S_{RB}\subseteq (S)_{RB}$, so $(S)_{RB}=( S_{RB})_{RB}$
whence $I=(S)_{RB}$, as required. The second statement follows since any element in $\sha_\bfk(\bfk)$ is a linear combination of homogeneous elements.
\end{proof}

Now we can give the following characterization of homogeneous Rota-Baxter ideals in $\sha_\bfk(\bfk)$.
\begin{prop}\label{prophom1RBI}
Let $I\subseteq \sha_{\bfk}(\bfk)$ be a Rota-Baxter ideal. Then $I$ is a homogeneous Rota-Baxter ideal if only if $I$ is a homogeneous ideal.
\end{prop}
\begin{proof}
($\Longleftarrow$) If a set of homogeneous elements generates $I$ as an ideal, then it does so as a Rota-Baxter ideal.
\smallskip

\noindent
($\Longrightarrow$) Suppose that $I$ is a homogeneous Rota-Baxter ideal. Then we can assume that there are a subset $\Lambda$ of $\NN$ and numbers $n_i\in \PP$ ($i\in \Lambda$) or $n_i=\infty$ such that $I$ is generated by the following set of homogeneous elements:
$$\mathcal{G}:=\{b_{ij}\fraka_i\in\bfk\fraka_i\,|\,j\in [n_i], i\in \Lambda\}.$$
Then, by Lemma \ref{rbigensetlem}, $I$ is the ideal generated by the set
$$
\mathcal{G}':=\left\{(\circ_{k=1}^mP_{x_{k}})(b_{ij}\fraka_i)\,|\, b_{ij}\fraka_i\in \mathcal{G}, x_{k}\in \sha_{\bfk}(\bfk)\ {\rm\ homogeneous}, 1\leq k\leq m, m\geq 0\right\}
$$
with the convention that $(\circ_{i=1}^0P_{x_{i}})(a):=a$.
Hence $I$ is contained in the homogeneous ideal generated by the supports of elements in $\mathcal{G}'$.
Thus to complete the proof, we just need to prove that for each $u\in \mathcal{G}'$, we have $\mathrm{supp}(u)\subset I$. We will prove this for  $u$ in the form $(\circ_{k=1}^mP_{x_{k}})(b_{ij}\fraka_i)$, where $b_{ij}\fraka_i\in \mathcal{G}$, by induction on $m\geq 0$. When $m=0$, we have $u=b_{ij}a_i\in \mathcal{G}$ which is assumed to be in $I$. Suppose that the statement is true for $m=k\geq 0$ and consider
$$u=(\circ_{k=1}^{m+1}P_{x_{k}})(b_{ij}\fraka_i)=
P_{x_{m+1}}\big((\circ_{k=1}^{m}P_{x_{k}})(b_{ij}\fraka_i)\big).$$
By the induction hypothesis, the support of $(\circ_{k=1}^{m}P_{x_{k}})(b_{ij}\fraka_i)$, that is, the set of homogenous components of it, is contained in $I$. Let $b\fraka_p\in I, b\in \bfk,$ be such a homogenous component and denote $x_{m+1}=c \fraka_n, c\in \bfk.$
Then by Eq. \eqref{productformula1}, we have
$$
P_{x_{m+1}}(b\fraka_p)=P\big((c\fraka_n) (b\fraka_p)\big)=P\left(\sum_{i=0}^{{\rm min}(n,p)}\binom{n+p-i}{p}\binom{p}{i}\lambda^ibc\fraka_{n+p-i}\right)
=\!\!\!\sum_{i=0}^{{\rm min}(n,p)}\binom{n+p-i}{p}\binom{p}{i}\lambda^icP^{n-i+1}(b\fraka_{p}).
$$
Since $P^{n-i+1}(b\fraka_{p})$ is in $I$, the homogeneous components of $P_{x_{m+1}}(b\fraka_p)$ are in $I$. Hence the homogeneous components of $u$ are in $I$. This completes the induction.
\end{proof}

\begin{theorem}\label{corhomoideal}
\begin{enumerate}
\item
Let $I$ be a homogeneous Rota-Baxter characteristic, that is, a homogeneous Rota-Baxter ideal of $\sha_\bfk(\bfk)$, with $A(I)=\left\{\left. (s_j,\Omega_{s_j})\,\right|\, j\in [N_I]\right\}$.
Then
\begin{equation}
I=\bigoplus_{i=s_1}^\infty \Omega_i \fraka_i =\bigoplus_{j=1}^{N_I}\left(\bigoplus_{i=s_{j}}^{s_{j+1}-1}\Omega_{s_{j}} \fraka_{i}\right)
=\left\{\begin{array}{ll}
\bigoplus\limits_{j=1}^\infty \bigoplus\limits_{i=s_j}^{s_{j+1}-1} \Omega_{s_j}\fraka_i, & N_I=\infty, \\
\left(\bigoplus\limits_{j=1}^{N_I-1} \bigoplus\limits_{i=s_j}^{s_{j+1}-1} \Omega_{s_j}\fraka_i\right)\bigoplus \left(\bigoplus\limits_{i=s_{N_I}}^\infty \Omega_{s_{N_I}}\fraka_i\right), & N_I<\infty,
\end{array}\right.
\label{eq:i}
\end{equation}
as the direct sum of the $\bfk$-modules $\Omega_{s_j} \fraka_{i}$. \label{it:homo1}
\item
The quotient $\sha_{\bfk}(\bfk)/I$ is the direct sum of the $\bfk$-modules $(\bfk/\Omega_{s_j})\fraka_i$, that is,
\begin{eqnarray}
\sha_{\bfk}(\bfk)/I&\cong& \bigoplus_{j=0}^{N_I}
\left(\bigoplus_{i=s_{j}}^{s_{j+1}-1}({\mathbf{k}}/\Omega_{s_{j}}) \fraka_{i}\right) \notag\\
&=&\left\{\begin{array}{ll}
\left(\bigoplus\limits_{i=0}^{s_1-1} \bfk\fraka_i\right) \bigoplus
\left(\bigoplus\limits_{j=1}^\infty \bigoplus\limits_{i=s_j}^{s_{j+1}-1} (\bfk/\Omega_{s_j})\fraka_i\right), & N_I=\infty, \\
\left(\bigoplus\limits_{i=0}^{s_1-1} \bfk\fraka_i\right) \bigoplus
\left(\bigoplus\limits_{j=1}^{N_I-1} \bigoplus\limits_{i=s_j}^{s_{j+1}-1} (\bfk/\Omega_{s_j})\fraka_i\right)\bigoplus \left(\bigoplus\limits_{i=s_{N_I}}^\infty (\bfk/\Omega_{s_{N_I}})\fraka_i\right), & N_I<\infty,
\end{array}\right.
\label{inhomdecompzzash1}
\label{rbida5}
\end{eqnarray}
with the convention that $s_{0}=0$, $\Omega_{s_0}=0$ and  $s_{N_I+1}=\infty$ if $N_I$ is finite.
\label{it:homo2}
\item
For any ascent pair $\{(s_j,\Omega_{s_j})\}_{j\in [N]}$ in $\cala$ defined in Eq.~\eqref{eq:ascpair}, consisting of a positive integer $N$ or $N=\infty$, strictly increasing nonnegative integers $s_i$ and strictly increasing ideals
$\Omega_{s_j}, j\in [N]$, there is a unique homogeneous Rota-Baxter ideal $I$ of $\sha_{\bfk}(\bfk)$ such that $A(I)=\{(s_j,\Omega_{s_j})\}_{j \in [N]}$. In other words,
the restriction of the map $\Phi:\cali\to \cala$ in Eq.~(\ref{eq:phi}) to the set of homogeneous Rota-Baxter characteristics is a bijection.
\label{it:homo3}
\end{enumerate}
\end{theorem}

\begin{proof}
(\ref{it:homo1})
By Proposition \ref{prophom1RBI}, the $\bfk$-module $I$ is the direct sum of the $\bfk$-modules $\Omega_{i} \fraka_{i}$.
\smallskip

\noindent
(\ref{it:homo2}) Since $I$ is a graded submodule of $\sha_{\bfk}(\bfk)$, Item~(\ref{it:homo2}) is proved.
\smallskip

\noindent
(\ref{it:homo3})
For a given ascending pair $\{(s_j,\Omega_{s_j})\}_{j\in [N]}\in \cala$, define $I\subseteq \sha_\bfk(\bfk)$ as in Eq.~(\ref{eq:i}) with $N_I$ replaced by $N$.
We next show that $I$ is a homogeneous Rota-Baxter ideal.

It follows from $s_1<s_2<\cdots$  and
$\Omega_{s_1}\subsetneq \Omega_{s_2}\subsetneq\cdots$ that $P(I)\subseteq I$. $I$ is a $\bfk$-submodule since $\Omega_{s_j}$ is an ideal of $\bfk$ for each
$j\in [N]$. Further by its definition, $I$ is homogeneous as a $\bfk$-module. We next prove that
$I(\sha_{\bfk}(\bfk))\subseteq I$.
Since each element  of $\sha_{\bfk}(\bfk)$ can be written as the summand of finitely many homogeneous components,
it suffices to show that $(b\fraka_{n})(c\fraka_{p})\in I$, where $b\fraka_{n}\in I$ and $c\fraka_{p}\in\sha_{\bfk}(\bfk)$.
By Eq.~\eqref{productformula1},   we have
\begin{align*}
\left(b\fraka_{n}\right)\left( c\fraka_{p}\right)
=\sum\limits_{i=0}^{{\rm min}(n,p)}\binom{n+p-i}{p}\binom{p}{i}bc\lambda^i\fraka_{n+p-i}.
\end{align*}
Since $b\fraka_{n}\in I$, we have $b\in\Omega_{s_r}$ where $r\in[N]$ is such that $s_r\leq n<s_{r+1}$.
For any given $i$ with $0\leq i\leq {\rm min}(n,p)$, let $t$ be the integer in $[N]$ such that $s_t\leq n+p-i<s_{t+1}$.
Note that we always have $n+p-i\geq n\geq s_r$, so $s_r\leq s_t$. Thus $\Omega_{s_r}\subseteq \Omega_{s_t}$ whence $b\in \Omega_{s_t}$.
Then for each $i$ with $0\leq i\leq {\rm min}(n,p)$, the element
$\binom{n+p-i}{p}\binom{p}{i}bc\lambda^i\fraka_{n+p-i}$ is in $\Omega_{s_t}\fraka_{n+p-i}.$
This shows that $\left(b\fraka_{n}\right)\left( c\fraka_{p}\right)\in I$. Thus $I$ is a homogeneous Rota-Baxter ideal of $\sha_{\bfk}(\bfk)$.

Thus we obtain a map $\Psi$ from $\cala$ to the set of homogeneous Rota-Baxter ideals of $\sha_\bfk(\bfk)$. It is direct to check that the left and right compositions of $\Psi$ with the restriction of $\Phi$ to homogeneous Rota-Baxter ideals of $\sha_\bfk(\bfk)$ are the identities. Thus the restriction of $\Phi$ to the set of homogeneous Rota-Baxter ideals is bijective.
\end{proof}

\subsection{Rota-Baxter characteristics and their quotients}
\label{ss:general}

Now we relate a Rota-Baxter characteristic to the homogeneous case and describe the structure of the quotient of $\sha_\bfk(\bfk)$ modulo a Rota-Baxter characteristic.
We will see that such a quotient has the underlying set of the form defined by Eq.~\eqref{inhomdecompzzash1}.
However, the same underlying set may be shared by different Rota-Baxter characteristics, as can be seen from the following theorem.
As mentioned at the beginning of this section, this theorem is a property of $\sha_{\bfk}(\bfk)$ coming as an analog to polynomial algebras.

We first relate a Rota-Baxter ideal of $\sha_\bfk(\bfk)$ to a suitable homogeneous Rota-Baxter ideal.

\begin{defn}
The {\bf initial Rota-Baxter ideal} of a Rota-Baxter ideal $I$ of $\sha_{\bfk}(\bfk)$ is the Rota-Baxter ideal ${\rm in}(I)$ generated by
$\left\{{\rm in}(f)|f\in I\right\}$.
\end{defn}

Thus ${\rm in}(I)$ is a homogeneous ideal.

\begin{prop}\label{inid}
Let $I$ be a Rota-Baxter ideal of $\sha_\bfk(\bfk)$. Then for its ascent set we have $A(I)=A({\rm in}(I))$. In other words, two Rota-Baxter characteristics are in the same fiber under the map $\Phi$ in Eq.~(\ref{eq:phi}) if and only if they have the same initial ideal.
\end{prop}
\begin{proof}
Assume that $A(I):=\left\{\left . (s_j,\Omega_{s_j})\,\right|\, j\in [N_I]\right\}$.
We only need to show that $\Omega_{i}(I)=\Omega_{i}({\rm in}(I))$ for all $i\in \NN$.
If $i<\td(I)$, then there is no element $f\in I$ with degree lower than $i$, so $\Omega_i(I)=\Omega_{i}({\rm in}(I))=\emptyset$.
Next we assume that $i\geq \td(I)$.
Take a nonzero element $b\in \bfk$. Then $b\in \Omega_{i}(I)$ if and only if there exists an element $f$ in $I$ such that ${\rm in}(f)=b\fraka_{i}$, which
is equivalent to $b\in \Omega_{i}({\rm in}(I))$. Thus $A(I)=A({\rm in}(I))$, as required.
\end{proof}

Now we can determine the underlying set of the quotient of a Rota-Baxter characteristic.

\begin{theorem}\label{mainthmdcomq}
Let $I$ be a Rota-Baxter ideal of $\sha_{\mathbf{k}}(\mathbf{k})$.
Then $\sha_{\mathbf{k}}(\mathbf{k})/I$ has the same underlying set as $\sha_{\mathbf{k}}(\mathbf{k})/{\rm in}(I)$ in Eq.~(\ref{inhomdecompzzash1}).
More precisely, for each $j$ with $j\in\{0\}\cup[N_I]$, fix a complete set $T_j\subseteq \bfk$, such that $0\in T_j$, of representatives of $\bfk$ modulo $\Omega_{s_j}$, with the convention that $s_{0}=0$, $\Omega_{s_0}=0$ and  $s_{N_I+1}=\infty$ if $N_I$ is finite.
Then $\sha_{\bfk}(\bfk)/I$ has a complete set of representatives given by the following subset of $\sha_\bfk(\bfk)=\bigoplus\limits_{m=0}^\infty \bfk \fraka_m$:
\begin{equation}
\mathcal{T}:=\bigoplus\limits_{j=0}^{N_I}
\left(\bigoplus\limits_{i=s_{j}}^{s_{j+1}-1}T_j \fraka_{i}\right)
=\left\{\begin{array}{ll}
\left(\bigoplus\limits_{i=0}^{s_1-1} \bfk\fraka_i\right) \bigoplus
\left(\bigoplus\limits_{j=1}^\infty \bigoplus\limits_{i=s_j}^{s_{j+1}-1} T_j\fraka_i\right), & N_I=\infty, \\
\left(\bigoplus\limits_{i=0}^{s_1-1} \bfk\fraka_i\right) \bigoplus
\left(\bigoplus\limits_{j=1}^{N_I-1} \bigoplus\limits_{i=s_j}^{s_{j+1}-1} T_j\fraka_i\right)\bigoplus \left(\bigoplus\limits_{i=s_{N_I}}^\infty T_{N_I}\fraka_i\right), & N_I<\infty.
\end{array}\right.
\label{decompzzash2}
\end{equation}
\end{theorem}

\begin{proof}
Let $I$ be a Rota-Baxter ideal of $\sha_{\mathbf{k}}(\mathbf{k})$ with its ascent set $ A(I)=\left\{ (s_j,\Omega_{s_j})\,|\,j\in [N_I]\right\}$.
By Proposition~\ref{inid}, we have $A(I)=A({\rm in}(I))$.
Thus it suffices to show that the underlying set of
$\sha_{\bfk}(\bfk)/I$ is $\mathcal{T}$ in light of Theorem~\ref{corhomoideal}\eqref{it:homo2}.

By convention, we take $s_0=0$ and $\Omega_{s_0}=0$.
Then $T_0=\bfk$ and
$$\bigoplus\limits_{i=0}^{s_1-1}{T_0} \fraka_{i}=\bigoplus\limits_{i=s_0}^{s_1-1}\mathbf{k} \fraka_{i}.$$
Note that  this term does not exist if $s_1=0$.

Let $f\in\sha_{\mathbf{k}}(\mathbf{k})$ be a nonzero element with ${\rm in}(f)= b_n \fraka_{n}$ for some $b_n\in\bfk$. So $\deg f=n\geq 0$.
We first show that  there is a unique $f'$ in $\mathcal{T}$ such that $f-f'$ is in $I$.

Let $\bigcup\limits_{j\in [N_I]}\left\{\left. f_{s_j,\ell}\,\right|\,  \ell\in [N_j]\right\}$ be an ascent generating set of $I$,
we may assume that for each $j\in [N_I]$ and $\ell\in [N_j]$ we have ${\rm in}(f_{s_j,\ell})=\omega_{s_j,\ell}\fraka_{s_j}$ and hence $\Omega_{s_j}$ is generated by the set
$\Theta_j:=\left\{\omega_{s_j,\ell}\,\left|\,\ell\in [N_j]\right.\right\}$ by Lemma \ref{generatingset2.3}.

If $\deg f<s_1$, then $f\in \bigoplus\limits_{i=0}^{s_1-1}\mathbf{k}\fraka_{i}$, which equals to $\bigoplus\limits_{i=s_0}^{s_1-1}T_0 \fraka_{i}$
since $s_0=0$, $\Omega_{s_0}=0$.
Thus, it is enough to take $f'=f$.

Next we assume that $\deg f\geq s_1$. We prove that there exists an element
$f'\in \mathcal{T}$ with $\deg f'\leq \deg f$ such that $f-f'\in I$ by induction on $\deg f$.
If $\deg f=s_1$, then ${\rm in}(f)=b_{s_1}\fraka_{s_1}$ for some $b_{s_1}\in\bfk$.
Since $T_1\subseteq \bfk$ is a complete set of representatives of $\bfk$ modulo $\Omega_{s_1}$,
there exists a unique $b_{s_1}'\in T_1$ such that $b_{s_1}-b_{s_1}'\in \Omega_{s_1}$.
Note that $\Omega_{s_1}$ is generated by $\Theta_1$, so there exist $c_{{s_1},\ell}$ in $\mathbf{k}$,
$\ell\in [N_1]$, with all but a finite number of $c_{{s_1},\ell}$ being zero, such that $b_{s_1}-b_{s_1}'=\sum\limits_{\ell\in [N_1]}c_{{s_1},\ell}\omega_{s_1,\ell}$.
If we take $f'=f-\sum\limits_{\ell\in [N_1]}c_{{s_1},\ell}f_{s_1,\ell}$, then $f-f'\in I$.
It is clear that $f'\in \mathcal{T}$ if $f'=0$. If $f'\neq0$, then $f'=b_{s_1}'\fraka_{s_1}+g'$ for some $g'\in \sha_{\bfk}(\bfk)$ with $\deg g'<\deg f=s_1$. Thus $g'\in \bigoplus\limits_{i=0}^{s_1-1}\mathbf{k}\fraka_{i}=\bigoplus\limits_{i=s_0}^{s_1-1}T_0 \fraka_{i}$ so that $f'$ is also in $\mathcal{T}$.

Now we assume that the claim has been proved for $f$ with $\deg f\leq n-1$ for a given $n\geq1$ and show that the claim is true when $\deg f= n$.
Assume that ${\rm in}(f)=b_n\fraka_n$.
Let $r\in  \{0\}\cup [N_I]$ be such that $s_r\leq n<s_{r+1}$.
So $\Omega_{n}=\Omega_{s_r}$ by Eq. \eqref{idealasscend}, and hence $b_n-b_n'\in\Omega_{s_r}$ for some $b_n'\in T_r$.
Then there exist $c_{n,\ell}\in \bfk$, $\ell\in [N_r]$, all but finitely many of which being zero, such that
$b_n-b_n'=\sum\limits_{\ell\in [N_r]}c_{n,\ell}\omega_{s_r,\ell}$.
Taking
$g=f-P^{n-s_r}(\sum\limits_{\ell\in [N_r]}c_{n,\ell} f_{s_r,\ell})$,
then we have $g=b_{n}'\fraka_{n}+g'$ for some $g'\in\sha_{\bfk}(\bfk)$ with $\deg g'<\deg f$.
If $g'=0$, then put $f'=b_{n}'\fraka_{n}$ so that $f'\in \mathcal{T}$ and we have
$f-f'\in I$.
If $g'\neq 0$, then  $\deg g'<\deg f=n$.
By the induction hypothesis, there exists $g''\in\mathcal{T}$ with $\deg g''\leq \deg g'$ such that
$g'-g''\in I$.
If we take
$f'=b_{n}'\fraka_{n}+g''$, then $f-f'\in I$.
Since $\deg g''\leq \deg g'$, we have $\deg g''<\deg f=n$,
so that $f'$\,¡¡is an element of $\mathcal{T}$.

Suppose that there is another element $h'$ in $\mathcal{T}$ such that $h'\neq f'$ and $f-h'\in I$.
Then $f'-h'\in I$. Let $f'=\sum\limits_{i=0}^kc_i\fraka_i$, $h'=\sum\limits_{i=0}^mb_i\fraka_i$,
where $b_i,c_i \in T_{r}$ if $s_r\leq i<s_{r+1}$ for some
$r\in \{0\}\cup[N_I]$. By symmetry, we may assume that $k\leq m$ and put $c_i=0$ for all $i$ with $k< i\leq m$.
Thus $h'-f'=\sum\limits_{i=0}^m(b_i-c_i)\fraka_i$.
Since $h'\neq f'$, there is a nonnegative integer $p$ with $0\leq p\leq m$ such that $b_p-c_p\neq0$ and ${\rm in}(h'-f')=(b_p-c_p)\fraka_p$.
Let $t\in \{0\}\cup[N_I]$ be such that $s_t\leq p<s_{t+1}$. Then $b_p\in T_t$, $c_p\in T_t$ and $b_p-c_p\in \Omega_{s_t}$.
But  $T_t$ is a complete set of representatives of $\bfk$ modulo $\Omega_{s_t}$, so $b_p=c_p$, a contradiction.

It remains to show that every element of $\mathcal{T}$ represents an element of $\sha_{\bfk}(\bfk)/I$.
For any $f'$ in $\mathcal{T}$ we take $f=f'$ in $\sha_{\bfk}(\bfk)$,
then $f-f'\in I$ whence $f'\in \mathcal{T}$ is corresponding to the element $f+I$ in $\sha_{\bfk}(\bfk)/I$. Thus, the underlying set of $\sha_{\mathbf{k}}(\mathbf{k})/I$ is $\mathcal{T}$ and the statement follows at once.
\end{proof}

We remark that the underlying $\bfk$-module $\mathcal{T}$ in Eq.~\eqref{decompzzash2} is usually not the direct sum of these $\bfk$-modules $({\mathbf{k}}/\Omega_{s_{j}}) \fraka_{i}$ for a nonhomogeneous Rota-Baxter ideal $I$, even though we write it in the form of a direct sum. The scalar multiplication by $\bfk$ should be defined according to the $\bfk$-module
$\sha_{\bfk}(\bfk)/I$. We illustrate this by the following example.

\begin{exam}\label{examplesetmod}
{\em Let $\sha_{\ZZ}(\ZZ)$ be the free Rota-Baxter $\ZZ$-algebra on $\ZZ$ of weight $1$. Let $I$ be the Rota-Baxter ideals of $\sha_{\ZZ}(\ZZ)$ generated by the element $f=2(\fraka_{1}+\fraka_{0})$. Then $\sha_{\ZZ}(\ZZ)/I$ and $\sha_\ZZ(\ZZ)/{\rm{in}}(I)$ are in bijection as sets, but are not isomorphic as $\ZZ$-modules (that is, abelian groups).
}
\end{exam}

\begin{proof}
Since ${\rm{in}}(f)=2\fraka_1$, we have $2a_{m+1}=P^{m}(2\fraka_1)\in {\rm{in}}(I)$ for
 $m\geq0$ whence ${\rm{in}}(I)\supseteq\bigoplus\limits_{i\geq1}2\ZZ\fraka_i$.

We note from Lemma \ref{rbigensetlem} that all elements of $I$ can be obtained from iterated operations on $f$ by
the Rota-Baxter operator $P$, the scalar product, the  multiplication and the addition of the algebra $\sha_{\ZZ}(\ZZ)$.
On one hand, we have
$P^m(f)=2(\fraka_{m+1}+\fraka_{m})$. On the other hand, it follows from Eq.~\eqref{productformula1} that
\begin{align*}
\fraka_{m}f=2(m+1)(\fraka_{m+1}+\fraka_{m}).
\end{align*}
Thus, any element of $I$ must be of the form
\begin{align}\label{form1I2}
2c_1(\fraka_{1}+\fraka_{0})+2c_2(\fraka_{2}+\fraka_{1})+\cdots
+2c_{n+1}(\fraka_{n+1}+\fraka_{n}),
\end{align}
where $n\in\NN$, $c_i\in\ZZ$, $1\leq i\leq n+1$. Consequently, ${\rm{in}}(I)\subseteq\bigoplus\limits_{i\geq1}2\ZZ\fraka_i$
and hence ${\rm{in}}(I)=\bigoplus\limits_{i\geq1}2\ZZ\fraka_i$.

Clearly, $A({\rm in}(I))=\{(1, 2\ZZ)\}$, whence both the underlying sets of $\sha_{\ZZ}(\ZZ)/{\rm in}(I)$ and $\sha_{\ZZ}(\ZZ)/I$ are
$$
\ZZ\fraka_0\bigoplus \left(\bigoplus\limits_{i\geq1}\ZZ_2\fraka_i\right)
$$
by Theorems~\ref{corhomoideal} and~\ref{mainthmdcomq}.

For the $\ZZ$-module $\sha_{\ZZ}(\ZZ)/{\rm in}(I)$, it follows from $2\fraka_1\in {\rm in}(I)$ that $\bar{2\fraka_1}=0$. However,
for the $\ZZ$-module $\sha_{\ZZ}(\ZZ)/I$, we have $2\fraka_1+I=-2\fraka_0+I$ so that $\bar{2\fraka_1}=\bar{-2\fraka_0}$, which is clearly not zero since
$-2\fraka_0$ is not in $I$ by Eq.~\eqref{form1I2}.
Therefore,  the Rota-Baxter ideals $I$ and ${\rm in}(I)$ share the same quotient sets, but not $\ZZ$-modules. From the fact that
$\bar{2\fraka_1}\neq 0$ in $\sha_{\ZZ}(\ZZ)/I$ we also see that $\sha_{\ZZ}(\ZZ)/I$ is not the direct sum of the $\ZZ$-modules
$\ZZ\fraka_0$ and $\ZZ_2\fraka_i$, $i\geq1$.
\end{proof}

The following example shows that a similar phenomenon can already be found in the polynomial algebra $\bfk[x]$.

\begin{exam}
Let $I$ be the ideal of $\ZZ[x]$ generated by the polynomial $2x+2$. Then the initial ideal ${\rm in}_{\leq}(I)$ is generated by $2x$. Thus the ideal $I$ and its initial ideal ${\rm in}(I)$ have the same quotient set
$$
\ZZ\bigoplus\left(\bigoplus_{n\geq1}\ZZ_2x^n\right).
$$
However, the two ideals do not have isomorphic quotient groups. For instance, $2x+I=-2+I$ is not zero since $-2$ is not in $I$, but $2x+{\rm in}_{\leq}(I)=0$ since
$2x\in {\rm in}_{\leq}(I)$. Furthermore,  $2x+I\neq 0$ shows that $\ZZ[x]/I$ is not the direct sum of the abelian groups $\ZZ$ and $\ZZ_2x^n$, $n\geq1$.
\end{exam}

\section{Characteristics of Rota-Baxter rings}\label{k=z3}

We next focus on the case when $\bfk=\ZZ$ and classify the Rota-Baxter ideals and prime Rota-Baxter ideal of $\sha_{\mathbb{Z}}(\mathbb{Z})$. Note that $\sha_{\mathbb{Z}}(\ZZ)$ is the initial object in the category of unitary Rota-Baxter rings (that is, Rota-Baxter $\ZZ$-algebras). So we are talking about characteristics of Rota-Baxter rings.

\begin{theorem}\label{mainthmdcom2}
Let $I$ be a Rota-Baxter ideal of $\sha_{\mathbb{Z}}(\mathbb{Z})$.
Then there exist a positive integer $m$ and $m$ pairs $(s_1,\omega_1)$, $\cdots$, $(s_m,\omega_m)\in \NN\times \PP$ with $s_j<s_{j+1}$, $\omega_{j+1}\neq \omega_{j}$ and $\omega_{j+1}|\omega_{j}$, $1\leq j\leq m-1$, such that $A(I)=\left\{(s_1,\omega_1\ZZ),\cdots,(s_m,\omega_m\ZZ)\right\}$.
Thus if $I$ is homogeneous then
$$
I=\bigoplus_{j=1}^m\left(\bigoplus_{i=s_j}^{s_{j+1}-1} \omega_{s_j}\ZZ\fraka_{i}\right)
$$
and the quotient  $\sha_{\mathbb{Z}}(\mathbb{Z})/I$ is isomorphic to
\begin{equation}\label{decompzzash1pz}
\bigoplus\limits_{j=0}^m
\left(\bigoplus\limits_{i=s_{j}}^{s_{j+1}-1}\mathbb{Z}_{\omega_j} \fraka_{i}\right),
\end{equation}
with the convention that $s_0=0$, $\omega_{s_0}=0$ and $s_{m+1}=\infty$. If $I$ is nonhomogeneous then $\sha_{\mathbb{Z}}(\mathbb{Z})/I$ has the same underlying set as $\sha_{\bfk}(\bfk)/{\rm in}(I)$ in Eq.~\eqref{decompzzash1pz}.

Conversely, for any positive integer $m$  and $m$ pairs $(s_1,\omega_1)$, $\cdots$, $(s_m,\omega_m)\in \NN\times \PP$ with $s_j<s_{j+1}$, $\omega_{j+1}\neq \omega_{j}$ and $\omega_{j+1}|\omega_{j}$, $1\leq j\leq m-1$,
there is a Rota-Baxter ideal $I$ of $\sha_{\ZZ}(\ZZ)$ such that the underlying set of $\sha_{\ZZ}(\ZZ)/I$ is the construction defined by  Eq.~\eqref{decompzzash1pz}.
\end{theorem}

\begin{proof}
Let $I$ be a Rota-Baxter ideal of $\sha_{\mathbb{Z}}(\mathbb{Z})$.
Since $\ZZ$ satisfies the ascending chain condition on ideals, $I$ has only finite ascend steps, say $s_1,\cdots,s_m$, where $m$ is a positive integer.
Since $\ZZ$ is a PID, we have $\Omega_{s_j}=\omega_j\ZZ$ for some positive integer $\omega_j$, $j=1,\cdots,m$.
Then, by Lemma \ref{idealisxi}, we have  $\omega_{j+1}\neq \omega_{j}$ and $\omega_{j+1}|\omega_{j}$, $1\leq j\leq m-1$. It follows from
the definition of $A(I)$ that $A(I)=\left\{(s_1,\omega_1\ZZ),\cdots,(s_m,\omega_m\ZZ)\right\}$. Consequently, in view of Theorems~\ref{corhomoideal} and~\ref{mainthmdcomq},
the first part of the theorem follows.

Conversely, take $\Omega_{s_j}=\omega_j\ZZ$. Then it follows from $s_j<s_{j+1}$, $\omega_{j+1}\neq \omega_{j}$ and $\omega_{j+1}|\omega_{j}$, $1\leq j\leq m-1$
that $\Omega_{s_1}\subsetneq\Omega_{s_2}\subsetneq\cdots\subsetneq\Omega_{s_m}$.
So by Theorem \ref{mainthmdcomq}, there is a Rota-Baxter ideal $I$ of $\sha_{\ZZ}(\ZZ)$ such that the underlying set of $\sha_{\ZZ}(\ZZ)/I$ is the construction defined by  Eq.~\eqref{decompzzash1pz}.
\end{proof}

Let $\bfk$ be an integral domain with characteristic $0$.
The next lemma tells us that if $\Omega_t=\bfk\omega_t$ is a principal ideal for $t=\td(I)$, then the element $f\in I$ with ${\rm in}(f)=\omega_t\fraka_{t}$ is completely determined by $\omega_t$ and the number of nonzero terms in $f$.

\begin{prop}\label{idealconskk}
Let $\bfk$ be an integral domain with characteristic $0$ and $\sha_{\bfk}(\bfk)$ the free Rota-Baxter algebra of weight $\lambda$, and let $I\subseteq \sha_{\bfk}(\bfk)$ be a nonzero Rota-Baxter ideal. Suppose that $\td(I)=t$ and
$\Omega_t=\bfk \omega_t$.
\begin{enumerate}
\item\label{l1a1f1} If $f=\sum\limits_{i=r}^{t}c_i\fraka_{i}\in I$ with $c_t=\omega_t$ and $c_r\neq0$, then
\begin{align*}
c_i=\binom{t-r}{t-i}\lambda^{t-i}c_t, \qquad r\leq i\leq  t-1.
\end{align*}

\item\label{hom1f2} If $\lambda=0$, then $I\subseteq \bigoplus\limits_{i\geq t}\bfk\fraka_{i}$.

\item\label{hom1f3} If $\lambda=0$ and $\bfk$ is a field, then $I=\bigoplus\limits_{i\geq t}\bfk\fraka_{i}$.
\end{enumerate}
\end{prop}

\begin{proof}
\eqref{l1a1f1} Since $I$ is a Rota-Baxter ideal and $f=\sum\limits_{i=r}^{t}c_i\fraka_{i}$ is in $I$, we have $g_1:=(t+1)P(f)-\fraka_{1}f\in I$.
Now from Eq.~\eqref{productformula1} it follows that
\begin{align*}
\fraka_{1}f
=(t+1)c_t\fraka_{t+1}+ \sum\limits_{i=r+1}^{t}i(c_{i-1}+\lambda c_{i})\fraka_{i}
+\lambda r c_{r}\fraka_{r},
\end{align*}
which together with
\begin{align*}
(t+1)P(f)=(t+1)\sum\limits_{i=r}^{t}c_{i}\fraka_{i+1}
=(t+1)\sum\limits_{i=r+1}^{t+1}c_{i-1}\fraka_{i}
\end{align*}
implies that
\begin{align*}
g_1=\sum\limits_{i=r+1}^{t}\left((t+1-i)c_{i-1}-\lambda ic_{i}\right)\fraka_{i}-\lambda rc_{r}\fraka_{r}.
\end{align*}
Thus, the coefficient of $\fraka_{t}$ is in $\Omega_t$, that is, $c_{t-1}-\lambda tc_{t}$ is in $\bfk c_t$ so that
$c_{t-1}$ is in $\bfk c_t$. Hence there exists $b\in\bfk$ such that $c_{t-1}=bc_t$.
Then
\begin{align*}
g_2:=&g_1-(b-\lambda t)f=\sum\limits_{i=r+1}^{t}\left((t+1-i)c_{i-1}+(\lambda (t-i)-b)c_i\right)\fraka_{i}+(\lambda (t-r)-b)c_r\fraka_{r}
\end{align*}
is in $I$.
But the coefficient of $\fraka_{t}$ in $g_2$ is $0$, so the fact that $\td(I)=t$ yields that
$g_2=0$, which is equivalent to $(\lambda (t-r)-b)c_r=0$  and
\begin{align}\label{eqsf}
(t+1-i)c_{i-1}+(\lambda (t-i)-b)c_i=0,\qquad r+1\leq i\leq t.
\end{align}
By the hypotheses that $\bfk$ is an integral domain and $c_r\neq0$, it follows from $(\lambda (t-r)-b)c_r=0$ that $b=\lambda (t-r)$.
Consequently, Eq. \eqref{eqsf} becomes
\begin{align*}
(t+1-i)c_{i-1}=\lambda(i-r)c_i,\qquad r+1\leq i\leq t.
\end{align*}
So for each $r\leq i\leq t-1$, we have
\begin{align*}
\prod_{j=i+1}^t\left((t+1-j)c_{j-1}\right)=\prod_{j=i+1}^t\left(\lambda(j-r)c_j\right),
\end{align*}
which is equivalent to
\begin{align*}
(t-i)!\prod_{j=i+1}^tc_{j-1}=\lambda^{t-i}\frac{(t-r)!}{(i-r)!}\prod_{j=i+1}^tc_j.
\end{align*}
Since $\bfk$ is  an integral domain with characteristic $0$, we obtain
\begin{align*}
c_i=\lambda^{t-i}\frac{(t-r)!}{(t-i)!(i-r)!}c_t=\binom{t-r}{t-i}\lambda^{t-i}c_t,\qquad r\leq i\leq t-1.
\end{align*}
This completes the proof of part \eqref{l1a1f1} of this lemma.
\smallskip

\noindent
\eqref{hom1f2} Take arbitrary $f\in I$.
Since $\td(I)=t$, we may assume that
\begin{align*}
f=b_m\fraka_{m}+b_{m-1}\fraka_{m-1}+\cdots+b_t\fraka_{t}+f_0\in I,
\end{align*}
where $m\geq t$ and  $\deg\ f_0\leq t-1$.
Note that $\Omega_t=\bfk\omega_t$. So there exists an element $g$ in $I$ with $\omega_t\fraka_t$ as the initial term.
Since $\lambda=0$, by part \eqref{l1a1f1}, we must have that $g=\omega_t\fraka_{t}$ is in $I$. So
$\omega_t\fraka_{m}=P^{m-t}(\omega_t\fraka_{t})$ is in $I.$
Then $h:=\omega_tf-b_m\omega_t\fraka_{m}$ is in $I$, that is,
\begin{align}\label{arfx1}
h=b_{m-1}\omega_t\fraka_{m-1}+\cdots+ b_t\omega_t\fraka_{t}+\omega_tf_0
\end{align}
is in $I$. Since $P(I)\subseteq I$, we have $\omega_t\fraka_{\ell}=P^{\ell-t}(\omega_t\fraka_{t})\in I$ for all $t\leq \ell\leq m-1$,
so that $b_{m-1}\omega_t\fraka_{m-1}+\cdots+ b_t\omega_t\fraka_{t}\in I$. Then Eq.~\eqref{arfx1} implies that
$\omega_tf_0\in I$.
Note that $\deg\ f_0\leq t-1$, which together with $\td(I)=t$ yields $\omega_tf_0=0$. But $\bfk$ is an integral domain, so $f_0=0$.
Therefore, $I\subseteq \bigoplus\limits_{i\geq t}\bfk\fraka_{i}$ and we are done.
\smallskip

\noindent
\eqref{hom1f3} Since $\lambda=0$, we know that $\omega_t\fraka_{t}$ is in $I$. But $\bfk$ is a field, so $\Omega_t=\bfk$,
whence $\bigoplus\limits_{i\geq t}\bfk\fraka_{i}\subseteq I$, which together with part \eqref{hom1f2} gives
$I= \bigoplus\limits_{i\geq t}\bfk\fraka_{i}$.
\end{proof}

We finally give a classification for the prime Rota-Baxter ideals of $\sha_{\ZZ}(\ZZ)$ when $\lambda=0$.
Recall that a  Rota-Baxter ideal $I$ of a Rota-Baxter algebra $(R,P)$ is said to be {\bf prime} if it is a prime ideal with $P(I)\subseteq I$.
If $\lambda=0$, then for any $m,n\in \NN$, it follows from Eq. \eqref{productformula1} that
\begin{align}\label{productformula0}
\fraka_{m} \fraka_{n}
=\binom{m+n}{m}\fraka_{m+n}.
\end{align}

\begin{theorem}\label{primesirb}
Let $\sha_{\ZZ}(\ZZ)$ be the free Rota-Baxter $\ZZ$-algebra on $\ZZ$ of weight $0$ and let $I$ be a proper nonzero Rota-Baxter ideal of $\sha_{\ZZ}(\ZZ)$. Then $I$ is prime if and only if
$$I=p\ZZ\fraka_{0} \bigoplus \left(\bigoplus\limits_{i\geq 1}\ZZ\fraka_{i}\right),$$
where $p$ is either $0$ or a prime number of $\ZZ$.
\end{theorem}

By Theorems \ref{mainthmdcom2} and \ref{primesirb}, the quotient of a prime characteristic of a Rota-Baxter ring is either $\ZZ$ or $\ZZ/p\ZZ$ for a prime number $p$, as in the case of prime characteristic of a ring.

\begin{proof}
Let $I$ be a prime Rota-Baxter ideal. Denote $t=\td(I)$ and $\Omega_t=\omega\ZZ$ for some positive integer $\omega$.
Then, by Proposition~\ref{idealconskk}\eqref{l1a1f1}, the element in $I$ of the form
$$\omega\fraka_{t}+ \text{ lower degree terms}$$
must be
$\omega\fraka_{t}$, that is, $\omega\fraka_{t}\in I$.

If $t\geq 1$, then $\omega \fraka_{t-1}\not\in I$ since $\td(I)=t$. From Eq.~\eqref{productformula0} we know that $\fraka_{1}(\omega\fraka_{t-1})=t \omega\fraka_{t}\in I$.
Since $I$ is prime, we have $\fraka_{1}\in I$, which means that $t={\rm st}(I)\leq1$ and hence $t=1$. Now $\fraka_{1}\in I$ gives $1\in \Omega_1$ so that
$\Omega_1=\ZZ$. By Lemma \ref{idealisxi}, $\Omega_j=\ZZ$ for all positive integer $j$. Therefore, $$I=\bigoplus\limits_{i\geq 1}\ZZ\fraka_{i}.$$

If $t=0$, then $\omega \fraka_{0}\in I$.
Since $I$ is a prime ideal and $\fraka_0$ is the identity of $\sha_{\ZZ}(\ZZ)$, $\omega$ must be a prime number. Let $\omega =p$.
Then $p\geq2$ and $p\fraka_{0}\in I$.
So
$\fraka_1^p=p! \fraka_p=(p-1)!P^p(p\fraka_0)\in I$ which means that $\fraka_1$ is in the prime ideal $I$.
Hence $I$ is generated by $p\fraka_{0}$ and $\fraka_{1}$.
Therefore, $$I=p\ZZ\fraka_{0} \bigoplus \left(\bigoplus\limits_{i\geq 1}\ZZ\fraka_{i}\right).$$

Conversely, $I$ is a homogeneous Rota-Baxter ideal generated by either $\{\fraka_1\}$ or $\{p\fraka_0,\fraka_1\}$,
where $p$ is a prime number. By Theorem \ref{mainthmdcom2}, if $I$ is generated by $\left\{\fraka_{1}\right\}$, then  $\sha_{\ZZ}(\ZZ)/I$ is isomorphic to $\ZZ$. If $I$ is generated by  $\left\{p\fraka_{0},\fraka_{1}\right\}$, then $\sha_{\ZZ}(\ZZ)/I$ is isomorphic to $\ZZ_p$.
In either case, $I$ is a prime ideal.
\end{proof}

\noindent {\bf Acknowledgements}:
This work was supported by NSFC grant $11426183$, $11501467$,
Chongqing Research Program of Application Foundation and Advanced Technology $($No. cstc2014jcyjA00028$)$.
H. Yu would like to thank the  NYU Polytechnic School of Engineering
for hospitality and support.


\begin{thebibliography}{99}

\bibitem{AL} M. Aguiar, On the associative analog of Lie bialgebras. {\em J. Algebra} {244} (2001), 492-532.

\bibitem{AGKO} G.~E. Andrews, L. Guo, W. Keigher and K. Ono, Baxter algebras and Hopf algebras,
{\em Trans. Amer. Math. Soc.}  355 (2003), 4639-4656.

\bibitem{Bai} C. Bai, A unified algebraic approach to classical
Yang-Baxter equation, {\em J. Phys. A: Math. Theor.}  40  (2007), 11073-11082.

\bibitem{Baibguoni2013}
C. Bai, O. Bellier, L. Guo and X. Ni, Spliting of operations, Manin products and Rota-Baxter operators,
{\em Int. Math. Res. Not.} 3 (2013), 485-524.

\bibitem{Baxter1960}
G. Baxter, An analytic problem whose solution follows from a simple algebraic identity, {\em Pacific J. Math.}
10 (1960), 731-742.

\bibitem{CK1998}
A. Connes and D. Kreimer, Hopf algebras, renormalization and noncommutative geometry,
{\em Commun. Math. Phys.} 199 (1998), 203-242.

\bibitem{C1972} P. Cartier, On the structure of free Baxter algebras, {\em Adv. Math.} 9 (1972), 253-265.

\bibitem{CLO98}
D. Cox, J. Little and D. O'shea, Ideals, Varieties, and Algorithms, Springer, 1992.

\bibitem{EG} K. Ebrahimi-Fard and L. Guo, Free Rota-Baxter algebras and dendriform algebras, {\em J. Pure Appl.
Algebra} { 212} (2008), 320-339.

\bibitem{GG}
N.~S. Gu and L.~Guo, Generating functions from the viewpoint of Rota-Baxter algebras. {\em Discrete Math.} { 338} (2015), 536–554.

\bibitem{guo2012}
L. Guo, An Introduction to Rota-Baxter Algebra, International Press (US) and Higher Education Press (China), 2012.

\bibitem{Gust} L. Guo, Baxter algebras, Stirling numbers and partitions, {\em J. Alg. Appl.} { 4} (2005), 153-164.

\bibitem{guo2009}
L. Guo, WHAT IS a Rota-Baxter algebra? {\em Notice Amer. Math. Soc.} 56 (2009), 1436-1437.

\bibitem{guokeigher20001}
L. Guo and W. Keigher, Baxter algebras and shuffle products, {\em Adv. Math.} 150 (2000), 117-149.

\bibitem{guozhang2008}
L. Guo and B. Zhang, Renormalization of multiple zeta values, {\em J. Algebra.} 319 (2008), 3770-3809.

\bibitem{Ro1} G.-C. Rota, Baxter algebras and combinatorial identities I\,\&\,II, {\em Bull. Amer. Math. Soc.} 75
(1969), 325-329, 330-334.

\bibitem{Ro2} G.-C. Rota,  Baxter operators, an introduction, In: ``Gian-Carlo Rota on Combinatorics, Introductory papers and commentaries", Joseph P.S. Kung, Editor,Birkh\"{a}user, Boston, 1995.

\bibitem{Ro1998}
G.-C. Rota, Ten mathematics problems I will never solve, Invited address at the joint meeting
of the American Mathematical Society and the Mexican Mathematical Society, Oaxaca,
Mexico, December 6, 1997, {\em Mitt. Dtsch. Math.-Ver., Heft} 2 (1998), 45-52.

\bibitem{STS} M. Semonov-Tian-Shansky, What is a classical
R-matrix?  {\em Funct. Anal. Appl.} 17 (1983), 259-272.

\end{thebibliography}
\end{document}